\documentclass[11pt]{article}

\usepackage[margin=1.05in]{geometry}
\usepackage{amsmath,amssymb,amsthm,mathtools}
\usepackage{booktabs}
\usepackage{enumitem}
\usepackage{float}
\usepackage[hidelinks]{hyperref}
\usepackage{url}
\usepackage{tikz}
\usetikzlibrary{arrows.meta,calc,positioning,fit,backgrounds}

\newtheorem{theorem}{Theorem}[section]
\newtheorem{corollary}[theorem]{Corollary}
\newtheorem{proposition}[theorem]{Proposition}

\theoremstyle{definition}
\newtheorem{definition}[theorem]{Definition}
\newtheorem{remark}[theorem]{Remark}

\DeclareMathOperator{\rank}{rank}

\DeclareMathOperator{\im}{im}
\DeclareMathOperator{\cl}{cl}

\title{Realized Rank Certificates for Matchstick Frameworks and Insertion Edges}
\author{Mike Winkler\\[1mm]
        \small Fakult\"at f\"ur Mathematik, Ruhr-Universit\"at Bochum, Germany\\ \small mike.winkler@ruhr-uni-bochum.de
        }
\date{}

\begin{document}
\maketitle

\begin{abstract}
We give an exact, checkable rank-certificate method for realized planar unit-distance frameworks.  The method is motivated by Vogel's computations for matchstick graphs and by the insertion-edge tests used in the Matchstick Graphs Calculator.  Its algebraic core is independent of geometry.  A singular square matrix $M$ is replaced by a sparse perturbation $B=M+UCV^T$.  If $B$ is nonsingular and the inverse satisfies $V^TB^{-1}U=C^{-1}$, then the columns of $B^{-1}U$ and the rows of $V^TB^{-1}$ form bases of the right and left kernels of $M$, and the rank defect of $M$ is certified.  Applied to the equilibrium matrix of a planar framework, this gives finite exact certificates for self-stresses, infinitesimal motions, redundant edges, and candidate edges whose constraints are already forced by the realized framework.  The certificate data can be checked independently from the search that produced it, using exact matrix identities.  We emphasize that matchstick frameworks are not generic: unit distances, triangles, rhombi, and symmetries can change the realized rank.  The method therefore concerns the coordinate-dependent representation of a given drawing, not only the generic rigidity matroid of the abstract graph.
\end{abstract}



\section{Introduction}

A matchstick graph is a plane straight-line graph whose edges all have unit length and whose edges do not cross; this terminology follows the standard usage for matchstick configurations in the plane \cite{Harborth1994}.  Such graphs are rigid enough to be arithmetically constrained, but special enough that generic rigidity alone does not describe their actual behavior; this is precisely the regime in which realized framework matrices, rather than only abstract sparsity counts, are needed \cite{AsimowRoth1978,AsimowRoth1979}.  This makes them a useful test class for rank methods in planar rigidity and for computer-assisted graph construction.

A recurring problem in computations with regular matchstick graphs is local.  A partial unit-distance framework has been assembled, and the endpoints of a remaining bar are very close to unit distance.  One has to decide whether a small motion of the existing framework can adjust the remaining distance error, or whether the distance is already forced by the bars that are present.  This is not a question about the numerical size of one residual error.  It is a rank question for the realized equilibrium matrix of the framework.

Vogel's Matchstick Graphs Calculator implements a practical method for this problem.  Its documentation relates the number of insertion edges to the mobility by the count
\[
        \text{insertion edges}=\text{mobility}+3+\frac12\sum_{v\in V}(\deg(v)-4),
\]
and therefore, for connected $4$-regular graphs, insertion edges equal mobility plus three \cite{VogelSoftware}.  Vogel's Matheplanet article describes the underlying matrix idea: add sparse correction entries to a singular matrix, invert the corrected matrix, and read linear dependencies from the inverse if the transposed correction pattern reappears \cite{Vogel2016}.  The present paper extracts the exact algebraic statement behind that procedure and formulates it in the language of realized framework rigidity.

The distinction between generic and realized rigidity is essential.  The classical planar rigidity matroid is obtained from a generic placement and is characterized by Laman's theorem \cite{Laman1970}.  Servatius gives a standard derivation of the infinitesimal rigidity matrix from the differentiated edge-length equations and explains why, at a generic placement, the linear dependencies depend only on the abstract graph \cite{Servatius1987}.  A matchstick realization is not generic.  Equal lengths, equilateral triangles, rhombi, repeated algebraic heights, and symmetries may produce additional dependencies.  Our rank certificates therefore refer to the concrete matrix of the drawing.

The main algebraic result is a sparse correction certificate.  Let $M\in K^{N\times N}$ be singular.  Let $U,V\in K^{N\times d}$ have full column rank, let $C\in K^{d\times d}$ be nonsingular, and set
\[
        B=M+UCV^T.
\]
If $B$ is nonsingular and
\[
        V^TB^{-1}U=C^{-1},
\]
then the columns of $B^{-1}U$ form a basis of $\ker M$, the rows of $V^TB^{-1}$ form a basis of the left kernel of $M$, and $\rank M=N-d$.  Thus the corrected inverse is not merely a computational aid.  It is a finite certificate of the rank defect and of both nullspaces.

For a planar framework, the right kernel of the equilibrium matrix is the space of self-stresses, while the left kernel is the space of infinitesimal motions.  This converts the sparse correction certificate into tests for mobility, redundant existing edges, and candidate insertion edges.  A candidate edge is already forced at the infinitesimal level precisely when its column lies in the realized column span of the present framework.

The contribution of the paper is threefold.  First, it gives a proof-level form of Vogel's sparse inverse test.  Second, it separates generic rigidity from realized unit-distance rigidity.  Third, it states exactly what data must be retained from a computation in order to turn a numerical or programmatic rank decision into a checkable certificate.  Thus the output of the computation is not only a rank value, but a finite proof object whose verification is independent of the search procedure.  This viewpoint is closely related to the paradigm of certifying algorithms: the computation produces not only an output, but also a certificate that can be checked independently of the algorithm which found it \cite{McConnellMehlhornNaeherSchweitzer2011}.

The paper is organized as follows.  Section \ref{sec:frameworks} recalls the rigidity and equilibrium matrices.  Section \ref{sec:generic_realized} compares generic rigidity with realized unit-distance rigidity.  Section \ref{sec:certificate} proves the sparse correction certificate.  Section \ref{sec:example} gives Vogel's basic matrix example.  Section \ref{sec:framework_application} treats rectangular framework matrices.  Section \ref{sec:edge_tests} gives redundant-edge and candidate-edge tests.  Sections \ref{sec:algorithm} and \ref{sec:numerics} discuss algorithmic certification, exact certificate files, and rounded computations.

\section{Framework matrices}\label{sec:frameworks}

Let $G=(V,E)$ be a finite graph with $V=\{1,\ldots,n\}$ and $E=\{e_1,\ldots,e_m\}$.  A planar framework is a map
\[
        p:V\to \mathbb R^2, \qquad i\mapsto p_i=(x_i,y_i).
\]
It is a matchstick framework if $\|p_i-p_j\|=1$ for every edge $ij\in E$ and if the straight-line drawing has no crossings.  The matrix definitions do not require unit lengths.  All framework applications in this paper take place over subfields of $\mathbb R$.  Over fields of characteristic $2$, the usual differential construction of the rigidity matrix would require a separate treatment.

Choose an arbitrary orientation of each edge.  For an oriented edge $e=ij$, define a column $a_e\in \mathbb R^{2n}$ by
\[
        a_e(2i-1,2i)=(x_i-x_j,y_i-y_j),
        \qquad
        a_e(2j-1,2j)=(x_j-x_i,y_j-y_i),
\]
and all other entries zero.  The equilibrium matrix is
\[
        A(G,p)=\begin{pmatrix} a_{e_1} & \cdots & a_{e_m}\end{pmatrix}
        \in \mathbb R^{2n\times m}.
\]
The usual rigidity matrix is
\[
        R(G,p)=A(G,p)^T.
\]

Figure \ref{fig:onebar} records this local contribution.  The entire method in this paper is built from such columns.

\begin{figure}[H]
\centering
\begin{tikzpicture}[>=Latex,scale=1]
\coordinate (i) at (0,0);
\coordinate (j) at (2.9,1.2);
\fill (i) circle (2pt) node[below left] {$p_i=(x_i,y_i)$};
\fill (j) circle (2pt) node[above right] {$p_j=(x_j,y_j)$};
\draw[line width=0.9pt] (i) -- (j) node[midway,above left] {$e=ij$};
\draw[->,thin,red] ($(i)+(0.10,0.10)$) -- ++(0.42,0.24);
\node[align=center] at (1.45,-1.5) {one row of $R(G,p)$\\ and one column of $A(G,p)$};
\node[anchor=west] at (4.9,0.1) {$
 a_e=
 \begin{pmatrix}
 0\\ \vdots\\
 x_i-x_j\\
 y_i-y_j\\
 \vdots\\
 x_j-x_i\\
 y_j-y_i\\
 \vdots\\ 0
 \end{pmatrix}$};
\node[anchor=west,align=left,text width=3.9cm] at (8.0,0.1) {nonzero entries occur only in the coordinate rows $i_x$, $i_y$, $j_x$, $j_y$};
\end{tikzpicture}
\caption{The local contribution of one bar.  The edge $e=ij$ contributes one row to the rigidity matrix and, after transposition, one column to the equilibrium matrix.  Vogel's certificates are formulated in terms of linear dependencies among these columns.}
\label{fig:onebar}
\end{figure}
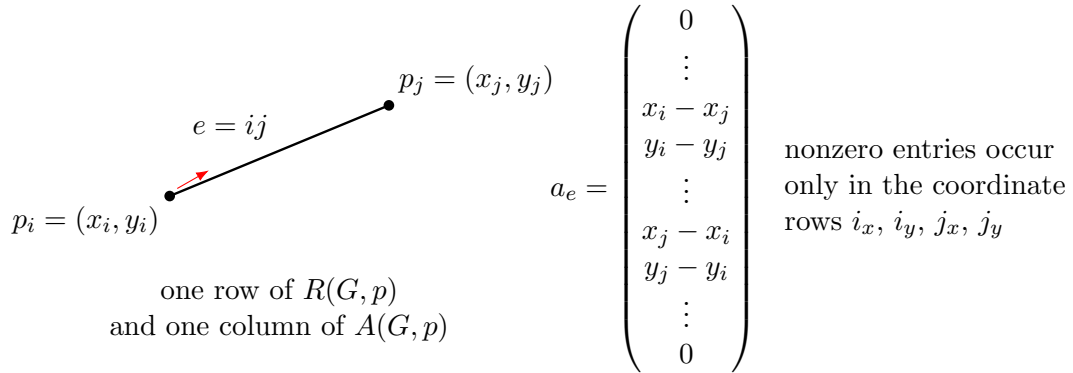

\begin{definition}
A vector $u\in\mathbb R^{2n}$ is an infinitesimal motion of $G(p)$ if
\[
        A(G,p)^T u=0.
\]
A vector $\omega\in\mathbb R^m$ is a self-stress of $G(p)$ if
\[
        A(G,p)\omega=0.
\]
\end{definition}

The terminology reflects the standard physical interpretation of bar-and-joint frameworks \cite{AsimowRoth1978,GraverServatiusServatius1993}.  The rigidity equation says that the relative velocity of the endpoints of every bar is perpendicular to the bar.  The equilibrium equation says that scalar tensions or compressions assigned to the bars balance at every vertex.

If $G(p)$ is connected and not collinear, the trivial infinitesimal motions have dimension three \cite{Servatius1987}.  They are the two translations and the rotation.  We define the internal infinitesimal mobility by
\[
        b(G,p)=\dim\ker A(G,p)^T-3.
\]
The stress dimension is
\[
        s(G,p)=\dim\ker A(G,p).
\]
Rank-nullity gives the Maxwell identity
\begin{equation}\label{eq:maxwell_general}
        s(G,p)-b(G,p)=m-2n+3.
\end{equation}
For a connected $4$-regular graph, $m=2n$, and therefore
\begin{equation}\label{eq:four_regular_count}
        s(G,p)=b(G,p)+3.
\end{equation}
For a graph with arbitrary degrees,
\[
        m-2n=\frac12\sum_{v\in V}(\deg(v)-4),
\]
so that \eqref{eq:maxwell_general} becomes
\begin{equation}\label{eq:degree_count}
        s(G,p)=b(G,p)+3+\frac12\sum_{v\in V}(\deg(v)-4).
\end{equation}
This is the counting rule used by the calculator, written in the notation of equilibrium matrices \cite{VogelSoftware}.

\begin{remark}
Equation \eqref{eq:degree_count} is only a dimension identity.  It does not decide rank.  Consequently, it always holds for the realized matrix, but it does not say whether the realized rank equals the generic upper bound or is smaller.  Once either mobility or stress dimension is known, the other follows.  The rank-certificate theorem below supplies the missing rank information.
\end{remark}

\section{Generic rigidity versus realized unit-distance rigidity}\label{sec:generic_realized}

The classical rigidity matroid is a matroid on the edge set of an abstract graph \cite{Crapo1979,GraverServatiusServatius1993}.  It is obtained by evaluating the rigidity matrix at a generic placement, or equivalently by using Laman's sparsity condition in the plane \cite{Laman1970,Crapo1979,GraverServatiusServatius1993}.  In this setting, independence depends only on the graph and not on a particular drawing.

Servatius gives a standard formulation of this point \cite{Servatius1987}.  A plane structure gives edge-length equations.  Differentiating these equations gives a homogeneous linear system $Hw=0$ for infinitesimal motions.  The rigid motions form a three-dimensional subspace of the infinitesimal motions.  At a generic placement, the relevant linear dependencies of the rows of $H$ depend only on the underlying graph, and Laman's theorem characterizes generic rigidity in the plane.

Matchstick frameworks are usually not generic, as indicated schematically in Figure \ref{fig:generic-realized}.  They contain unit edges, equilateral triangles, rhombi, repeated algebraic heights, and often prescribed symmetries.  Thus the realized matrix $A(G,p)$ may have additional dependencies not present in a generic placement of the same abstract graph.  Conversely, a statement proved only for the generic rigidity matroid may fail to describe a particular unit-distance drawing.

Figure \ref{fig:generic-realized} emphasizes the conceptual distinction used throughout the paper.  The abstract graph is the same, but the right-hand drawing lies in a special unit-distance position with additional geometric relations.

\begin{figure}[H]
\centering
\begin{tikzpicture}[>=Latex,scale=1]
\node at (1.6,2.45) {generic placement};
\coordinate (a1) at (0,-0.30);
\coordinate (a2) at (2.1,-0.25);
\coordinate (a3) at (2.75,1.25);
\coordinate (a4) at (0.45,1.65);
\draw (a1)--(a2)--(a3)--(a4)--cycle;
\draw (a1)--(a3);
\foreach \p/\lab/\pos in {a1/1/below left,a2/2/below right,a3/3/right,a4/4/left}
  {\fill (\p) circle (2pt) node[\pos] {$\lab$};}

\node at (7.0,2.45) {unit-distance placement};
\coordinate (b1) at (5.6,0.2);
\coordinate (b2) at (7.4,0.2);
\coordinate (b3) at (8.3,1.45);
\coordinate (b4) at (6.6,1.45);
\draw (b1)--(b2)--(b3)--(b4)--cycle;
\draw (b1)--(b3);
\draw[red,dotted,line width=1pt] (6.6,1.45) -- (6.6,0.2);
\foreach \p/\lab/\pos in {b1/1/below left,b2/2/below right,b3/3/right,b4/4/left}
  {\fill (\p) circle (2pt) node[\pos] {$\lab$};}
\node[align=center] at (6.85,-0.95) {equal lengths, symmetries, and\\ special algebraic relations};
\end{tikzpicture}
\caption{Generic placements and unit-distance placements encode different rank questions.  In a generic framework the rigidity properties are governed by the abstract graph alone, whereas a unit-distance realization may carry additional algebraic relations coming from equal lengths, equilateral triangles, rhombi, or symmetries.}
\label{fig:generic-realized}
\end{figure}
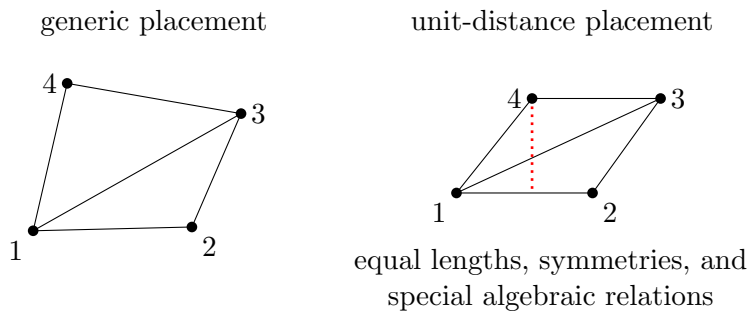

For this reason we use the following terminology.

\begin{definition}
Let $A=A(G,p)$ be the equilibrium matrix of a realized framework.  The realized edge matroid $\mathcal M_A$ is the representable matroid on the set $E$ in which a subset $F\subseteq E$ is independent if and only if the columns $\{a_e:e\in F\}$ are linearly independent.
\end{definition}

Its rank function is
\[
        r_A(F)=\rank A_F,
\]
where $A_F$ is the submatrix consisting of the columns indexed by $F$.  Its closure operator is
\[
        \cl_A(F)=\{e\in E: r_A(F\cup\{e\})=r_A(F)\}.
\]
For a candidate edge $c$ not in $E$, we use the same notation by adjoining its candidate column $a_c$ to the representation.

In a generic placement, $\mathcal M_A$ is the usual planar rigidity matroid restricted to $E$ \cite{Laman1970,Crapo1979,GraverServatiusServatius1993}.  In a matchstick placement, $\mathcal M_A$ is the realized specialization which must be used in exact computations.  Vogel's method is a method for computing this realized matroidal information, not merely the generic Laman rank \cite{Vogel2016,VogelSoftware}.

\begin{remark}
This distinction also explains why numerical drawings can be misleading.  Rounding a special unit-distance realization may destroy exact linear dependencies.  The rounded drawing then represents a different realized matroid.
\end{remark}

\section{A sparse rank-certificate theorem}\label{sec:certificate}

The following theorem is the algebraic core.  It is stated over an arbitrary field $K$.  This is useful because exact computations may take place over $\mathbb Q$, over an algebraic number field, or over a field generated by construction parameters.

\begin{theorem}[Sparse correction certificate]\label{thm:sparse_certificate}
Let $M\in K^{N\times N}$.  Let $U,V\in K^{N\times d}$ have full column rank, and let $C\in K^{d\times d}$ be nonsingular.  Put
\[
        B=M+UCV^T.
\]
Assume that $B$ is nonsingular and that
\begin{equation}\label{eq:inverse_block_condition}
        V^T B^{-1}U=C^{-1}.
\end{equation}
Then
\[
        X=B^{-1}U,
        \qquad
        Y^T=V^T B^{-1}
\]
satisfy
\[
        MX=0,
        \qquad
        Y^TM=0.
\]
Moreover, $\rank M=N-d$.  Hence the columns of $X$ form a basis of $\ker M$, and the rows of $Y^T$ form a basis of the left kernel of $M$.
\end{theorem}

\begin{proof}
Set $X=B^{-1}U$.  Since $B=M+UCV^T$, one has
\[
        M=B-UCV^T.
\]
Therefore
\[
        MX=(B-UCV^T)B^{-1}U
          =U-UC(V^TB^{-1}U)
          =U-UCC^{-1}=0.
\]
Similarly, with $Y^T=V^TB^{-1}$,
\[
        Y^TM=V^TB^{-1}(B-UCV^T)
            =V^T-(V^TB^{-1}U)CV^T
            =V^T-C^{-1}CV^T=0.
\]
It remains to prove the rank statement.  Write
\[
        M=B(I-XCV^T).
\]
The matrix $P=XCV^T$ is an idempotent, since
\[
        P^2=XCV^TXCV^T=XC(V^TX)CV^T=XCC^{-1}CV^T=XCV^T=P.
\]
Moreover, $\rank P=d$.  Indeed, $P=X(CV^T)$ gives $\rank P\le d$, while
\[
        PX=XCV^TX=XCC^{-1}=X
\]
and the columns of $X=B^{-1}U$ are independent.  Hence $\rank P\ge d$.  Thus $P$ is an idempotent of rank $d$, and $I-P$ has rank $N-d$.  Since $B$ is nonsingular, $\rank M=N-d$.  The nullspaces of $M$ and $M^T$ have dimension $d$.  The columns of $X$ and the rows of $Y^T$ are independent.  They are therefore bases.
\end{proof}

\begin{remark}
The full-column-rank hypothesis on $U$ and $V$ is not independent of
\eqref{eq:inverse_block_condition}.  If that identity holds, then
$V^TB^{-1}U=C^{-1}$ is nonsingular, and therefore
\[
        d=\rank(V^TB^{-1}U)\le \min(\rank U,\rank V),
\]
since $B$ is nonsingular.  Hence $\rank U=\rank V=d$.  We keep the
hypothesis in the statement because it makes the certificate easier to
read and because it is known before the inverse-block identity is checked
in most implementations.
\end{remark}

The most common case uses coordinate selection matrices
\[
        U=\begin{pmatrix}e_{i_1}&\cdots&e_{i_d}\end{pmatrix},
        \qquad
        V=\begin{pmatrix}e_{j_1}&\cdots&e_{j_d}\end{pmatrix}.
\]
If $C=I_d$, the correction consists of $d$ added ones in the positions $(i_t,j_t)$.  Condition \eqref{eq:inverse_block_condition} says that the transposed inverse block
\[
        \bigl((B^{-1})_{j_s i_t}\bigr)_{s,t=1}^d
\]
is the identity matrix.  Thus each marked entry reappears as a one in the transposed position of the inverse, and all cross entries vanish.  This is the exact form of the correction test described in Vogel's article \cite{Vogel2016}.

\begin{corollary}[One borrowed entry]\label{cor:one_entry}
Let $M\in K^{N\times N}$.  Suppose that
\[
        B=M+e_i e_j^T
\]
is nonsingular and that
\[
        (B^{-1})_{ji}=1.
\]
Then $\rank M=N-1$.  The vector $B^{-1}e_i$ spans $\ker M$, and the row $e_j^T B^{-1}$ spans the left kernel of $M$.
\end{corollary}

\begin{remark}
The theorem separates search from proof.  Pivoting, exchanges, and Woodbury updates may be used to find a useful correction.  The proof requires only the final matrix identity \eqref{eq:inverse_block_condition}, checked exactly.
\end{remark}

Figure \ref{fig:sparse-certificate} shows the schematic pattern of the certificate.  A sparse correction $P=UCV^T$ is added to a singular matrix.  The inverse of the corrected matrix displays the transposed pattern which certifies the rank defect and reveals bases of the right and left kernels.

\begin{figure}[H]
\centering
\begin{tikzpicture}[x=0.55cm,y=0.55cm,>=Latex]
\node at (1.25,5.95) {$M$};
\foreach \r in {0,...,5} {
  \draw (0,\r) -- (2.5,\r);
  \draw (0.5*\r,0) -- (0.5*\r,5);
}
\node at (3.35,2.5) {$+$};
\node at (5.15,5.95) {$P$};
\foreach \r in {0,...,5} {
  \draw (4.0,\r) -- (6.5,\r);
  \draw ({4.0+0.5*\r},0) -- ({4.0+0.5*\r},5);
}
\fill[red!28] (5.5,3) rectangle (6.0,4);
\fill[red!28] (4.5,1) rectangle (5.0,2);
\node[red!70!black] at (5.75,3.5) {$1$};
\node[red!70!black] at (4.75,1.5) {$1$};
\node at (7.2,2.5) {$=$};
\node at (9.1,5.95) {$B=M+P$};
\foreach \r in {0,...,5} {
  \draw (7.8,\r) -- (10.3,\r);
  \draw ({7.8+0.5*\r},0) -- ({7.8+0.5*\r},5);
}
\fill[red!28] (9.3,3) rectangle (9.8,4);
\fill[red!28] (8.3,1) rectangle (8.8,2);
\node[red!70!black] at (9.55,3.5) {$1$};
\node[red!70!black] at (8.55,1.5) {$1$};
\draw[->,thick] (10.8,2.5) -- (11.9,2.5);
\node at (14.9,5.95) {$B^{-1}$};
\foreach \r in {0,...,5} {
  \draw (12.3,\r) -- (16.8,\r);
  \draw ({12.3+0.9*\r},0) -- ({12.3+0.9*\r},5);
}
\fill[blue!22] (12.3,4) rectangle (13.2,5);
\fill[blue!22] (15.0,1) rectangle (15.9,2);
\node[blue!70!black] at (12.75,4.5) {$1$};
\node[blue!70!black] at (15.45,1.5) {$1$};
\node[align=center,text width=10.4cm] at (8.4,-1.15) {selected columns of $B^{-1}$ span $\ker M$,\\ selected rows of $B^{-1}$ span the left kernel of $M$};
\end{tikzpicture}
\caption{A sparse correction certificate.  A singular matrix $M$ is made invertible by adding a sparse correction $P$.  If the inverse of $M+P$ contains the corresponding transpose pattern, then the selected columns and rows of the inverse certify the rank defect of $M$.}
\label{fig:sparse-certificate}
\end{figure}
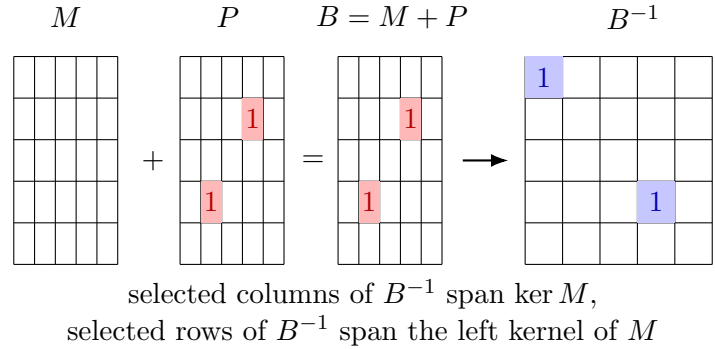

\section{The basic matrix example}\label{sec:example}

We include Vogel's introductory computation as a model \cite{Vogel2016}.  Let
\[
A=\begin{pmatrix}
7&1&13&-11\\
18&12&6&18\\
21&15&4&26\\
9&6&3&9
\end{pmatrix}.
\]
The matrix is singular.  Add one in position $(2,3)$ and set
\[
B=A+e_2e_3^T
 =\begin{pmatrix}
7&1&13&-11\\
18&12&7&18\\
21&15&4&26\\
9&6&3&9
\end{pmatrix}.
\]
Then
\[
B^{-1}=\begin{pmatrix}
7&0&25&-191/3\\
-15&-2&-54&425/3\\
0&1&0&-2\\
3&1&11&-30
\end{pmatrix}.
\]
The transposed entry is $(B^{-1})_{3,2}=1$.  Corollary \ref{cor:one_entry} gives
\[
        x=B^{-1}e_2=\begin{pmatrix}0\\-2\\1\\1\end{pmatrix},
        \qquad
        y^T=e_3^TB^{-1}=\begin{pmatrix}0&1&0&-2\end{pmatrix}.
\]
Indeed,
\[
        Ax=0,
        \qquad
        y^TA=0.
\]
Thus the corrected inverse contains both the column dependence and the row dependence of the original matrix.  Explicitly,
\[
        -2\begin{pmatrix}1\\12\\15\\6\end{pmatrix}
        +\begin{pmatrix}13\\6\\4\\3\end{pmatrix}
        +\begin{pmatrix}-11\\18\\26\\9\end{pmatrix}=0
\]
and
\[
        \begin{pmatrix}18&12&6&18\end{pmatrix}
        -2\begin{pmatrix}9&6&3&9\end{pmatrix}=0.
\]
The example also illustrates why the position of the added entry matters.  The added one is not arbitrary data in the final proof.  It is a device which makes the matrix invertible and whose transposed inverse entry certifies a rank defect of the uncorrected matrix.

\section{Application to rectangular framework matrices}\label{sec:framework_application}

The equilibrium matrix of a framework is usually rectangular.  The sparse correction theorem applies after a square completion.  The completion is only a bookkeeping device.  Artificial kernel directions introduced by zero rows or zero columns must be removed before interpreting the result geometrically.  The completions below are chosen because they preserve one of the two genuine kernels while making the matrix square.

\begin{remark}
For the $4$-regular graphs that motivate this paper no completion is
needed.  A connected $4$-regular graph has $m=2n$, so the equilibrium
matrix $A(G,p)\in\mathbb R^{2n\times 2n}$ is already square, and
Theorem \ref{thm:sparse_certificate} applies to $A$ directly: $\ker A$
is the space of self-stresses and $\ker A^T$ is the space of
infinitesimal motions.  This is the setting of Vogel's procedure, in
which a square matrix is corrected and inverted.  The square completions
below are required only for the general case in which the equilibrium
matrix is rectangular.
\end{remark}

Let $A\in K^{r\times m}$ be an exact equilibrium matrix.  If $r\ge m$, set
\[
        \widehat A=\begin{pmatrix} A&0\end{pmatrix}\in K^{r\times r}.
\]
The added zero columns create artificial right-kernel directions.  They are not self-stresses of the framework.  The left kernel is unchanged.  If $m\ge r$, set
\[
        \widehat A=\begin{pmatrix} A\\0\end{pmatrix}\in K^{m\times m}.
\]
The added zero rows create artificial left-kernel directions.  They are not framework motions.  The right kernel is unchanged.

\begin{proposition}[Extraction from a square completion]\label{prop:completion}
Let $A\in K^{r\times m}$ and let $\widehat A$ be one of the square completions above.  Suppose that Theorem \ref{thm:sparse_certificate} is applied to $\widehat A$.  The resulting null vectors form complete bases of $\ker\widehat A$ and $\ker\widehat A^T$.  After projecting away the components supported only in the added zero rows or columns, these bases span the genuine framework kernels $\ker A$ and $\ker A^T$; after zero vectors and redundancies are discarded, they yield bases of these kernels.  The dimensions of the two genuine kernels are always
\[
        \dim\ker A=m-\rank A,
        \qquad
        \dim\ker A^T=r-\rank A.
\]
\end{proposition}

\begin{proof}
For $\widehat A=(A\;0)$ the equation $\widehat A z=0$ reads $Az_0=0$, where $z_0$ is the part of $z$ in the original $m$ columns.  The remaining coordinates are artificial.  The equation $w^T\widehat A=0$ is exactly $w^TA=0$ on the original rows.  Thus the right kernel of $\widehat A$ is $\ker A$ plus the artificial zero-column directions, while the left kernel is $\ker A^T$.  The case $\widehat A=\begin{pmatrix}A\\0\end{pmatrix}$ is analogous, with artificial directions in the left kernel instead.  The dimension formulas are rank-nullity and do not depend on the certificate.  Since Theorem \ref{thm:sparse_certificate} gives complete bases of the two kernels of $\widehat A$, the projected families span the corresponding genuine kernels of $A$; selecting independent nonzero vectors gives bases.
\end{proof}

For a non-collinear connected planar framework $G(p)$, Proposition \ref{prop:completion} yields
\[
        s(G,p)=\dim\ker A(G,p),
        \qquad
        b(G,p)=\dim\ker A(G,p)^T-3.
\]
Together with \eqref{eq:degree_count}, this gives the relation between mobility and insertion-edge count used in the computational output.

\begin{remark}
The term insertion edge is not standard in rigidity theory; we use it here in the computational sense of Vogel's software documentation \cite{VogelSoftware}.  In the present formulation it is best understood in two related senses.  The number of insertion edges is the stress dimension predicted by \eqref{eq:degree_count}.  The displayed insertion-edge ranges are supports or closures in the realized edge matroid $\mathcal M_A$.  In Vogel's tables the symbol $A$ counts correction entries in a square completion.  In the $4$-regular setting motivating this paper, after the artificial directions from a completion have been discarded, this count agrees with the stress dimension $s=b+3$.  For arbitrary auxiliary examples it need not coincide with $s$, since a square completion may also contain entries belonging only to artificial rows or columns.
\end{remark}

\section{Redundant edges, closure, and candidate insertion tests}\label{sec:edge_tests}

Figure \ref{fig:candidate-edge} illustrates the closure point of view.  The drawn framework is chosen so that its columns already have full realized rank for its vertex set.  Hence the dashed candidate bar does not create a new infinitesimal constraint, even though it is absent from the current drawing.

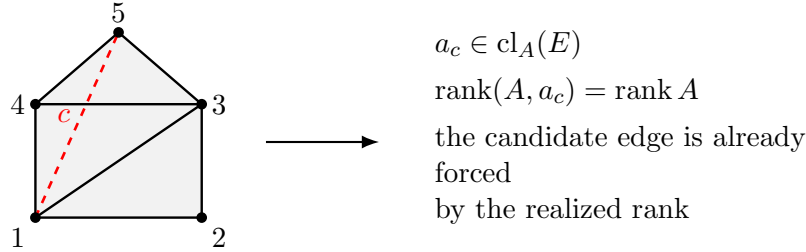
\begin{figure}[H]
\centering
\begin{tikzpicture}[>=Latex,scale=1]
\coordinate (p1) at (0.5,0);
\coordinate (p2) at (2.7,0);
\coordinate (p3) at (2.7,1.5);
\coordinate (p4) at (0.5,1.5);
\coordinate (p5) at (1.6,2.45);
\fill[gray!10] (p1)--(p2)--(p3)--(p5)--(p4)--cycle;
\draw[line width=0.9pt] (p1)--(p2)--(p3)--(p4)--cycle;
\draw[line width=0.9pt] (p4)--(p5)--(p3);
\draw[line width=0.9pt] (p1)--(p3);
\draw[red,dashed,line width=1pt] (p1)--(p5) node[pos=0.55,left] {$c$};
\foreach \p/\lab/\pos in {p1/1/below left,p2/2/below right,p3/3/right,p4/4/left,p5/5/above}
  {\fill (\p) circle (2pt) node[\pos] {$\lab$};}
\draw[->,thick] (3.55,1.0) -- (5.05,1.0);
\node[anchor=west,align=left,text width=5.8cm] at (5.65,1.2) {$a_c\in\cl_A(E)$\\[4pt] $\rank(A,a_c)=\rank A$\\[4pt] the candidate edge is already forced\\ by the realized rank};
\end{tikzpicture}
\caption{An insertion edge as a realized closure test.  The dashed candidate edge is redundant precisely when its column lies in the span of the columns of the existing framework.  Equivalently, adding it does not increase the realized rank.}
\label{fig:candidate-edge}
\end{figure}

The next proposition gives the exact meaning of an edge which can be removed without changing the infinitesimal motion space.

\begin{proposition}[Existing redundant edges]\label{prop:redundant_existing}
Let $A=A(G,p)$ be the equilibrium matrix, and let $e\in E$.  Let $A_{-e}$ be obtained from $A$ by deleting the column $a_e$.  The following are equivalent.
\begin{enumerate}[label=\textup{(\roman*)}]
\item There exists a self-stress $\omega\in\ker A$ with $\omega_e\ne0$.
\item The column $a_e$ lies in the span of the other columns of $A$.
\item $\rank A_{-e}=\rank A$.
\item $\ker A_{-e}^T=\ker A^T$.
\item $e\in\cl_A(E\setminus\{e\})$.
\end{enumerate}
Thus deleting $e$ does not change the infinitesimal motion space if and only if $e$ occurs with nonzero coefficient in a self-stress.
\end{proposition}

\begin{proof}
If $\omega\in\ker A$ and $\omega_e\ne0$, then
\[
        \omega_e a_e=-\sum_{f\ne e}\omega_f a_f,
\]
so $a_e$ is in the span of the other columns.  The converse is the same equation read backwards.  This proves the equivalence of (i) and (ii).  The equivalence of (ii), (iii), and (v) is the definition of rank and closure in the realized matroid.  Finally, two matrices with the same row number have equal column span if and only if the orthogonal complements of their column spans are equal.  These orthogonal complements are $\ker A^T$ and $\ker A_{-e}^T$.
\end{proof}

Candidate edges can be treated in the same representation.  Let $H=(V,E)$ be a realized framework and let $c=ij$ be a pair of vertices not necessarily in $E$.  Define the candidate column $a_c$ from the coordinates of $i$ and $j$ exactly as for an edge.  It is the transpose of the differential of the squared-distance constraint for $c$.

\begin{proposition}[Candidate edge test]\label{prop:candidate_edge}
Let $A=A(H,p)$.  For a candidate edge $c$, the following are equivalent.
\begin{enumerate}[label=\textup{(\roman*)}]
\item $a_c\in\im A$.
\item $c\in\cl_A(E)$ after adjoining $a_c$ to the realized representation.
\item Every infinitesimal motion $u\in\ker A^T$ satisfies $a_c^T u=0$.
\item Adding $c$ does not decrease the dimension of the infinitesimal motion space.
\item The enlarged equilibrium matrix $(A\;a_c)$ has one more independent self-stress than $A$.
\end{enumerate}
\end{proposition}

\begin{proof}
The equivalence of (i) and (ii) is the definition of closure.  The equivalence of (i) and (iii) is the identity
\[
        (\im A)^\perp=\ker A^T.
\]
If $a_c\in\im A$, then adding $a_c$ does not change the column span and hence does not change the left kernel.  This proves (iv).  Conversely, if the left kernel does not change, then the orthogonal complement of the column span does not change, so the column span does not change.  This gives (i).  The rank of $(A\;a_c)$ is the same as the rank of $A$ precisely in this case.  Since one column has been added, the right nullity increases by one.
\end{proof}

\begin{corollary}[Linear obstruction to adjustment]\label{cor:linear_obstruction}
If $a_c\in\im A(H,p)$, then the length of the candidate edge $c$ is stationary to first order along every infinitesimal motion of $H(p)$.  If, in addition, $H(p)$ is infinitesimally rigid, then the candidate distance is locally fixed up to Euclidean motion.  This last implication uses the standard fact that infinitesimal rigidity implies local rigidity for planar bar-and-joint frameworks \cite{AsimowRoth1978,AsimowRoth1979}.
\end{corollary}

\begin{remark}
Corollary \ref{cor:linear_obstruction} is deliberately infinitesimal.  The certificate does not replace nonlinear arguments.  For frameworks which are not infinitesimally rigid, a candidate edge may be stationary to first order but still vary along a finite flex in higher order.  Thus the certificate identifies constraints already settled at the linear level; remaining questions require higher-order analysis or finite-flex constructions.  A finite non-insertability statement requires infinitesimal rigidity of the relevant subframework or an additional nonlinear argument, in line with the standard distinction between infinitesimal and finite rigidity \cite{AsimowRoth1978,AsimowRoth1979}.
\end{remark}

\begin{remark}
Corollary \ref{cor:linear_obstruction} certifies that the length of a
candidate edge $c=ij$ is fixed to first order by the realized framework.
Under infinitesimal rigidity this becomes a local metric constraint.  It
does not certify that this length equals the length of any other vertex
pair.  Equality of two separately fixed distances is a metric coincidence
forced by symmetry or by the algebraic construction relations, not a
property of the realized edge matroid $\mathcal M_A$.  Thus a certificate
that $|p_i-p_j|$ is stationary does not by itself justify replacing it by
a value $|p_l-p_m|$ taken elsewhere in the drawing; such an identification
requires a separate exact relation $|p_i-p_j|=|p_l-p_m|$.
\end{remark}

\section{Algorithmic certification}\label{sec:algorithm}

The following form is suitable for exact computation, for a proof appendix, and for supplementary certificate data.  Its purpose is to separate the search for a rank decision from the verification of the final proof object.

\begin{enumerate}[label=\textup{\arabic*.}]
\item Construct the equilibrium matrix $A(G,p)$ over an exact field.  For a matchstick framework this may require algebraic coordinates or exact construction relations.
\item Replace $A$ by a square completion $\widehat A$ if necessary.
\item Find sparse matrices $U,V$ and an invertible small matrix $C$ such that
\[
        B=\widehat A+UCV^T
\]
is nonsingular.
\item Verify exactly that
\[
        V^TB^{-1}U=C^{-1}.
\]
This identity is the rank certificate.
\item Extract
\[
        X=B^{-1}U,
        \qquad
        Y^T=V^TB^{-1}.
\]
The columns of $X$ are right null vectors of $\widehat A$, and the rows of $Y^T$ are left null vectors.
\item Remove artificial components introduced by square completion.  Interpret the remaining right null vectors as self-stresses and the remaining left null vectors as infinitesimal motions.
\item Verify the final products
\[
        AX_0=0,
        \qquad
        Y_0^T A=0
\]
for the projected vectors $X_0,Y_0$.  These products are redundant if Step 4 is exact, but they should nevertheless be included in the certificate file.  They catch transcription, parsing, and indexing errors before the certificate is accepted.
\item For each marked edge set, record the support of the corresponding self-stress.  Proposition \ref{prop:redundant_existing} identifies redundant existing edges, and Proposition \ref{prop:candidate_edge} identifies candidate columns already forced by the framework.
\end{enumerate}

In an implementation, one need not recompute an inverse from scratch when the correction changes; this is also how the certificate should be separated from the search procedure in software implementations \cite{Vogel2016,VogelSoftware}.  The Woodbury identity is used between nonsingular corrected iterates, not by inverting the singular target matrix.  If $B_0$ is nonsingular, $B_0^{-1}$ is known, and
\[
        B_1=B_0+U'C'V'^T,
\]
then
\[
        B_1^{-1}=B_0^{-1}
        -B_0^{-1}U'(C'^{-1}+V'^TB_0^{-1}U')^{-1}V'^TB_0^{-1},
\]
whenever the displayed inverse exists.  This identity is an acceleration device for moving from one invertible corrected matrix to another.  It is not a substitute for the final certificate.  The purpose of the certificate is verifiability rather than asymptotic optimality; implementation-specific speedups should not replace the final exact identities.

\section{Exact arithmetic and rounded drawings}\label{sec:numerics}

The main numerical danger is a false rank decision.  Matchstick coordinates frequently contain algebraic numbers, for instance multiples of the height of an equilateral triangle.  If coordinate differences are rounded independently, identities such as
\[
        x_j-x_i=(x_j-x_k)+(x_k-x_i)
\]
may be destroyed.  Since the columns of $A(G,p)$ are built from such differences, this can destroy exact self-stresses and exact infinitesimal motions.

There are three computational regimes which should be kept separate.

\begin{enumerate}[label=\textup{(\alph*)}]
\item \emph{Exact certification.}  The matrix entries lie in an exact field and the inverse-block identity is checked exactly.  This is a proof.
\item \emph{Exact computation from rounded input.}  The rounded coordinates are treated as exact rational data.  This proves a statement about the rounded framework, not necessarily about the intended framework.
\item \emph{Tolerance-based recovery.}  During elimination a number $x$ with $|x|<\varepsilon$ is treated as zero.  This may recover the dependency structure intended by the construction, but it is not a proof until followed by exact reconstruction or an exact certificate.
\end{enumerate}

Vogel's documentation records these distinctions in computational form \cite{VogelSoftware}.  It lists exact GAP computations, exact browser computations with integer arithmetic, rounded computations, and an adjustable tolerance used when a pivot is close to zero.  This is a sound practical strategy if the output intended for publication is converted into an exact certificate.

For a proof-oriented data file, and in particular for supplementary certificate material, the following items are sufficient:
\[
        A,
        \quad U,V,C,
        \quad B^{-1}U,
        \quad V^TB^{-1},
        \quad V^TB^{-1}U=C^{-1}.
\]
Together with the definitions of the coordinate field and the projected check products $AX_0=0$ and $Y_0^TA=0$, these data certify the rank, the self-stress space, and the infinitesimal motion space.  A verifier need not reproduce the search, only the exact matrix identities and the coordinate relations used to construct $A$.

\section{Consequences for matchstick-graph constructions}

The rank-certificate method is especially useful in searches for minimal regular matchstick graphs, including the type of $4$-regular unit-distance constructions considered in \cite{WinklerDinkelackerVogel2019}.  Such searches often produce nearly completed frameworks with a small number of non-fitting edges.  The method answers two distinct questions.

First, it determines whether a subframework has internal infinitesimal mobility.  If the internal mobility is zero, then an exact non-collinear framework is locally rigid.  This gives a local obstruction to adjusting distances by a small motion.

Second, it locates self-stress supports and matroid closures in the realized representation.  If a candidate edge lies in the realized closure of the existing edge columns, then its constraint is already a linear consequence of the remaining constraints.  Conversely, if its column is independent of the existing columns, adding it removes one infinitesimal degree of freedom.  This is the linear algebra behind the insertion-edge display in the calculator.

The method is not a replacement for the nonlinear part of a construction.  It does not by itself prove that an approximate drawing can be moved to an exact unit-distance drawing.  It supplies the rank information needed to decide which local adjustments are possible and which edge constraints are already forced.  In this form it is a suitable tool for rigorous computational papers, because the final certificate consists of finite exact matrix equalities.

The scope of these conclusions is infinitesimal.  The certificate proves linear dependence in the realized rigidity representation.  It does not, without further hypotheses, prove existence or nonexistence of a finite flex.  In applications this distinction is useful rather than restrictive: the certificate identifies exactly which local constraints require nonlinear treatment and which ones are already settled at first order.

\section{Conclusion}

Vogel's computational procedure can be stated as an exact rank certificate for the realized equilibrium matrix of a matchstick framework.  The sparse correction is not an ad hoc alteration of the problem once the inverse-block identity is verified.  It proves that the original matrix has a prescribed rank defect and that the corrected inverse contains bases for both nullspaces.

The surrounding rigidity theory clarifies what the certificate proves.  In a generic placement it computes information in the usual planar rigidity matroid.  In a matchstick placement it computes information in the realized edge matroid of the given unit-distance drawing.  This realized setting is the one needed for small regular matchstick graphs and for local insertion-edge questions.

Applied to the equilibrium matrix, the right nullspace is the space of self-stresses and the left nullspace is the space of infinitesimal motions.  The method therefore provides a unified certificate for mobility, stress-supported edge dependencies, and local candidate-edge tests.  Its proper use requires a clear distinction between exact certification, exact computation on rounded input, and tolerance-based numerical recovery.  With this distinction made explicit, the method gives a rigorous foundation for further computational and theoretical work on regular matchstick graphs and for certifiable computations with realized planar frameworks.

\section*{Acknowledgments}

The author thanks Stefan Vogel for developing and documenting the computational method which motivated this article, and for making the corresponding software available.  Vogel's Matheplanet article and the Matchstick Graphs Calculator documentation are cited above as the primary sources for the underlying computational idea.  The terminology and presentation in the present paper were adapted to the standard framework-rigidity setting in order to make the resulting certificates independent of a particular implementation.

\end{document}